\documentclass{article}
\usepackage[british]{babel}
\usepackage{amsmath,amssymb,stmaryrd,enumerate,hyperref,cleveref,latexsym,theorem}

\newcommand{\assign}{:=}
\newcommand{\mathd}{\mathrm{d}}
\newcommand{\nin}{\not\in}
\newcommand{\of}{:}
\newcommand{\precpreceq}{\preceq\!\!\preceq}
\newcommand{\suchthat}{:}
\newcommand{\tmaffiliation}[1]{\\ #1}
\newcommand{\tmem}[1]{{\em #1\/}}
\newcommand{\tmemail}[1]{\\ \textit{Email:} \texttt{#1}}
\newcommand{\tmop}[1]{\ensuremath{\operatorname{#1}}}
\newcommand{\tmstrong}[1]{\textbf{#1}}
\newcommand{\tmtextit}[1]{\text{{\itshape{#1}}}}
\newenvironment{enumeratealpha}{\begin{enumerate}[a{\textup{)}}] }{\end{enumerate}}
\newenvironment{itemizeminus}{\begin{itemize} }{\end{itemize}}
\newenvironment{proof}{\noindent\textbf{Proof\ }}{\hspace*{\fill}$\Box$\medskip}
\newcounter{nnacknowledgments}

{\theorembodyfont{\rmfamily}\newtheorem{acknowledgments*}[nnacknowledgments]{Acknowledgments}}
\newtheorem{corollary}{Corollary}[section]
\newtheorem{definition}{Definition}[section]
{\theorembodyfont{\rmfamily}\newtheorem{example}{Example}[section]}
\newtheorem{lemma}{Lemma}[section]
\newtheorem{proposition}{Proposition}[section]
{\theorembodyfont{\rmfamily}\newtheorem{remark}{Remark}[section]}
\newtheorem{theorem}{Theorem}[section]
\newtheorem{maintheorem}{Theorem}

\begin{document}

\title{Formal conjugacy and asymptotic differential algebra}

\author{
  Vincent Bagayoko
  \tmaffiliation{IMJ-PRG (Paris)}
  \tmemail{bagayoko@imj-prg.fr}
}

\maketitle

\begin{abstract}
  We study conjugacy of formal derivations on fields of generalised power
  series in characteristic $0$ and dimension $1$. Casting the problem of
  Poincar{\'e} resonance in terms of asymptotic differential algebra, we give
  conditions for conjugacy of parabolic flat $\log$-$\exp$ transseries, flat
  grid-based transseries, logarithmic transseries, power series with exponents
  and coefficients in an ordered field, and formal Puiseux series.
\end{abstract}

\section*{Introduction}

Conjugating and normalising local analytic diffeomorphisms around fixed points
is a classical method for classifying dynamical systems. Its formal version,
where the conjugating element is a possibly divergent formal power series, is
usually easier to tackle, as it is devoid of convergence issues (see
{\cite{Ec:compensators}}). Yet normalising formal power series may be
difficult because of the phenomenon of resonance, as first studied by
Poincar{\'e} {\cite{Poincar�:vol1}} and Dulac {\cite{Dulac:phd}}. Resonance
introduces non-convexity into the conjugacy problem: given three series $f$,
$g$ and $h$ where $h$ is closer to $f$ than $g$ from a valuative standpoint,
it may be that $f$ and $g$ are conjugate whereas $f$ and $h$ are not (see
\Cref{subsection-resonnance}).

Our main motivation for this paper comes from our interest in first-order
properties of certain valued groups {\cite{Bag:gog,Bag:c-val}}. From the model
theoretic standpoint, the existence of resonance means that the geometry of
definable sets in those structures is too complicated to study. Thus finding
contexts in which conjugacy is resonance-free is crucial. We hope to convince
the reader that as far as formal normalisation of local objects is concerned,
the three following notions, belonging to seemingly disparate domains, are
strongly connected:
\begin{itemizeminus}
  \item non-resonance, as a linear algebraic condition for linearisation of
  vector fields {\cite{Ec:compensators}},
  
  \item convexity of conjugacy, as a case of tameness of definable sets in
  valued groups {\cite{Bag:c-val}},
  
  \item asymptotic integration, as a closure property for valued differential
  fields {\cite{Rosli80}}.
\end{itemizeminus}
In connection {\cite{GaKaSpei:Ilya,PeResRoSer:linear}} with Dulac's problem
(see {\cite{Du1923,Ilya85,Ec90}}), the dynamics of Poincar{\'e} first-return
maps, and the analysis of limit cycles of vector fields {\tmem{via}} Dulac
series, there has been interest recently
{\cite{PeResRoSer:normal,Peran:parabolic,Peran:hyperbolic}} in normalising
formal series which may involve formal exponentials and logarithms of the
infinite variable $x$. This is for instance notable in {\'E}calle's method of
linearisation by compensators {\cite{Ec:compensators}}. For these more general
questions, a natural domain of investigation is the field of
logarithmic-exponential transseries {\cite{DG87,Ec92,vdDMM01}}. It is known
how to normalise purely logarithmic transseries {\cite{Gre:phd}} of the form
$\lambda x + o (x)$, in the hyperbolic case {\cite{PeResRoSer:normal}}, i.e.
when $\lambda \neq 1$, and in the parabolic case {\cite{Peran:parabolic}},
when $\lambda = 1$. There is ongoing work on the hyperbolic and parabolic
cases in the more general setting of H-fields {\cite{AvdD02,AvdD03}} equipped
with composition laws (see {\cite[Section~4.1]{Bag:gog}}).

In this paper, we focus on the conjugacy problem for parabolic series in
differential valued fields {\cite{Rosli80}}. These include for instance formal
Laurent or Puiseux series, $\log$-$\exp$ transseries or grid-based transseries
{\cite{vdH:ln}}, logarithmic transseries {\cite{Gre:phd}}, or $\omega$-series {\cite{BM19}}, and complexifications thereof. This
choice is not fortuitous but motivated by the connections between the setting
of (ordered) asymptotic differential algebra and that of (ordered) valued
groups (see {\cite[Remark~7.27]{Bag:c-val}}). One of the difficulties of
solving conjugacy equations for formal series endowed with a composition law
$\circ$ and a derivation $\partial$ is that this requires a good understanding
of the interaction of the composition law with the valuation. On this path,
one is confronted with intricate and computationally heavy problems involving
Taylor expansions of arbitrarily high orders, monotonicity of the composition
law, and mean value inequalities. It is not the least of hindrances that such
properties of $\circ$ and $\partial$ may not have been established for the
given algebra of formal series.

We circumvent these issues by leveraging the Lie-type correspondence, given by
a formal exponential map $\exp$, between near-identity substitutions $f
\mapsto f \circ (x + \delta)$ and contracting derivations $f \mapsto g
\partial (f)$ on algebras of formal series. We showed {\cite{BKKMP:strong}}
that a fraction of the theory of Lie groups applies to such algebras. We focus
on contracting derivations on fields of generalised series, i.e. derivations which piecewise increase the valuation of non-trivial elements. On
the side of vector fields, contractiveness can be related to nilpotency (see {\cite{Martinet}}).

Given a direct limit $\mathbb{S}$ of fields of Hahn series with its natural
valuation $v$ and a derivation $\partial \of f \mapsto f'$ on $\mathbb{S}$
that is compatible with the structure of direct limit of fields of series (see
\Cref{def-regular}), we consider a group $(\tmop{Cont} (\partial), \ast)$
introduced in {\cite{Bag:c-val}} of contracting derivations on $\mathbb{S}$.
The group law $\ast$ is a formal Baker-Campbell-Hausdorff operation (see
{\cite{Mal'cev:Lie,Lazard:Lie,Serre:Lie,CiGraVa:Lie}}). This group also has a
structure of Lie algebra and can be seen as a linearisation of its ``Lie
group'' $\exp (\tmop{Cont} (\partial))$. The latter is a group, under
functional composition, of substitutions. In $(\tmop{Cont} (\partial), \ast)$,
finding approximate solutions of conjugacy equations reduces to finding
approximate solutions of linear differential equations of order $1$ (see
\Cref{lem-approximate-conjugacy}). Using spherical completeness arguments, one
can obtain exact solutions by transfinitely iterating the approximation method
(see \Cref{lem-exact}). This composition-free framework allows us to easily
understand obstructions to conjugacy, and to cast resonance merely as a
property of asymptotic differential algebra, i.e. as a property of the valued
differential field $(\mathbb{S}, v, \partial)$. This gives a simple connection
between features of the asymptotic couple {\cite[Section~9.1]{vdH:mt}} of
$\mathbb{S}$, in particular asymptotic integration {\cite[p 327]{vdH:mt}}, and
the existence of resonance for conjugacy equations and normal forms. Say that
$\mathbb{S}$ has regular asymptotic integration if asymptotic integration on
$\mathbb{S}$ is compatible with its structure of direct limit (see
\Cref{def-semi-regular}). We prove:

\begin{maintheorem}
  \label{th-1}{\tmem{[\Cref{th-main}]}} Suppose that $\mathbb{S}$ has regular
  asymptotic integration. Let $f, g \in \tmop{Cont} (\partial) \setminus \{ 0
  \}$. Then $f$ and $g$ are conjugate in $\tmop{Cont} (\partial)$ if and only
  if $v (f - g) > v (g)$ and $v (f - g) > v \left( \left( h / gh' \right)'
  \right)$ for all $h \in \mathbb{S}^{\times}$ with $v (h) > 0$.
\end{maintheorem}

Another benefit of our method is that it is independent of the field of
scalars $C$, in that the results are preserved under extensions of scalars
(see \Cref{rem-extension-of-scalars}). Using the Lie-type correspondence, we
obtain a more classical reformulation of the conjugacy problem:

\begin{maintheorem}
  {\tmem{[\Cref{th-main-standard}]}} Suppose that $\mathbb{S}$ has regular
  asymptotic integration. Let $f, g \in \tmop{Cont} (\partial) \setminus \{ 0
  \}$ such that $v (f - g) > v (f)$ and and $v (f - g) > v \left( \left( h /
  gh' \right)' \right)$ for all $h \in \mathbb{S}^{\times}$ with $v (h) > 0$.
  Then the derivations $f \partial \assign h \mapsto f \partial(h)$ and $g \partial\assign h \mapsto g \partial(h)$ are conjugate over
  $\tmop{Aut} (\mathbb{S})$, i.e. there is a $\sigma = \exp (h \partial) \in
  \tmop{Aut} (\mathbb{S})$ such that~$\sigma \circ (g \partial) \circ
  \sigma^{\tmop{inv}} = f \partial$.
\end{maintheorem}

In certain cases, the group of automorphisms $\exp (\tmop{Cont} (\partial))$
is isomorphic to a well-identified group of series under composition. Let
$\mathbb{S}$ be the field of transseries whose transmonomials $\mathfrak{m}$
satisfy $v (\mathfrak{m}' /\mathfrak{m}) + v (x) \geqslant 0$. We identify
(\Cref{prop-subsystems}) the group $\exp (\tmop{Cont} (\partial))$ for all
direct limits of subsystems (\Cref{def-subsystem}) of the direct system which
defines $\mathbb{S}$. Combining this with Theorem~\ref{th-1}, we obtain:

\begin{maintheorem}
  {\tmem{[\Cref{th-transseries}]}} For all $\delta, \varepsilon \in
  \mathbb{S}$ with $v (\delta), v (\varepsilon) > v (x)$, the series $x +
  \delta$ and $x + \varepsilon$ are conjugate in $\{ x + \rho \suchthat \iota
  \in \mathbb{S} \wedge v (\rho) > v (x) \}$ if and only if $v (\varepsilon -
  \delta) > v (\delta - x \delta')$.
\end{maintheorem}

We also recover (\Cref{cor-log-transseries}) the resonance-free part of
Peran's results {\cite[Corollary 2.4]{Peran:parabolic}} on the field
$\mathbb{T}_{\log}$ of logarithmic transseries {\cite{Gre:phd}}. In the
resonant case of formal power series with exponents in an ordered field, we
have a result (\Cref{th-k-powered}), and a counterexample
(\Cref{subsection-resonnance}) to the convexity of the conjugacy problem in
the non-resonant case (\Cref{cor-initial}).

\section{Groups of contracting derivations}

Throughout the paper, we fix a field $C$ of characteristic $0$, a non-empty
directed set $(D, \leqslant)$ and a directed system $\mathcal{S}=
(\Gamma_d)_{d \in D}$, for the inclusion, of non-trivial ordered Abelian
groups. We write $\Gamma$ for the direct limit of $\Gamma_d$.

Let $d \in D$. We have a field $\mathbb{S}_d \assign C \left( \! (\Gamma_d)
\! \right)$ of Hahn series {\cite{Hahn1907}} with coefficients in $C$ and
exponents in $\Gamma_d$. This is the ring, under pointwise sum and Cauchy
product, of functions $f \of \Gamma_d \longrightarrow C$ whose support
$\tmop{supp} f = \{ g \in \Gamma_d \suchthat f (g) \neq 0 \}$ is a
well-ordered subset of~$\Gamma_d$. It contains $C$ canonically, and we have
canonical inclusions $\mathbb{S}_{d_0} \longrightarrow \mathbb{S}_{d_1}$
whenever $d_0 \leqslant d_1$.

There is a notion of infinite sum for certain families in $\mathbb{S}_d$
called {\tmem{summable families}} (see {\cite[Section~3.1]{vdH:noeth}}), and
the corresponding structure is a summability algebra in the sense of
{\cite[Definition~1.27]{BKKMP:strong}}. Let $\mathbb{S}$ be the direct limit
of the directed system $(\mathbb{S}_d)_{d \in D}$. This is a summability
algebra for the direct limit summability structure, where a family is summable
if and only if it takes values in an $\mathbb{S}_{d_0}$, for a $d_0 \in D$, in
which it is summable. A linear map $\mathbb{S} \longrightarrow \mathbb{S}$
(resp. $\mathbb{S}_d \longrightarrow \mathbb{S}_d$) that commutes with
infinite sums is said {\tmem{strongly linear}}, and we write $\tmop{Lin}^+
(\mathbb{S})$ (resp. $\tmop{Lin}^+ (\mathbb{S}_d)$) for the algebra under
pointwise sum and composition of strongly linear maps on $\mathbb{S}$ (resp.
$\mathbb{S}_d$).

\begin{definition}
  \label{def-regular}A linear map $\phi \of \mathbb{S} \longrightarrow
  \mathbb{S}$ is said {\tmem{{\tmstrong{regular}}}} if for all $d \in D$ we
  have $\phi (\mathbb{S}_d) \subseteq \mathbb{S}_d$ and $\phi \upharpoonleft
  \mathbb{S}_d$ is strongly linear. We write $\tmop{Lin}^{\mathcal{S}}
  (\mathbb{S})$ for the set of regular linear maps $\mathbb{S} \longrightarrow
  \mathbb{S}$.
\end{definition}

Note that $\tmop{Lin}^{\mathcal{S}} (\mathbb{S}) \subseteq \tmop{Lin}^+
(\mathbb{S})$. We can see $\tmop{Lin}^{\mathcal{S}} (\mathbb{S})$ as a Lie
algebra for the Lie bracket $\llbracket \cdot, \cdot \rrbracket : (\phi, \psi)
\mapsto \phi \circ \psi - \phi \circ \psi$. The set $\tmop{Der}^{\mathcal{S}}
(\mathbb{S})$ of regular derivations on $\mathbb{S}$ is closed under
$\llbracket \cdot, \cdot \rrbracket$, thus it is a Lie algebra. It is also
closed under infinite sums {\cite[Proposition~2.2]{BKKMP:strong}}.

We have a valuation $v$ on $\mathbb{S}$ given by $v (f) = \min \tmop{supp} f
\in \Gamma$ for all $f \in \mathbb{S}^{\times}$ and~$v (0) = + \infty$. We
write $\preccurlyeq$ for the corresponding dominance relation
{\cite[Definition~3.3.1]{vdH:mt}}, given by $f \preccurlyeq g
\Longleftrightarrow v (f) \geqslant v (g)$ for all $f, g \in \mathbb{S}$. We
write $f \prec g$ if $v (f) > v (g)$, $f \asymp g$ if $v (f) = v (g)$ and $f
\sim g$ if $v (f - g) > v (f)$. We write $\mathbb{S}^{\prec} \assign \{ f \in
\mathbb{S} \suchthat f \prec 1 \}$ for the set of infinitesimal elements in
$\mathbb{S}$.

A linear map $\phi : \mathbb{S} \longrightarrow \mathbb{S}$ is said
{\tmem{contracting}} if $\phi (f) \prec f$ for all $f \in
\mathbb{S}^{\times}$. Given $d_0 \in D$, we write $\tmop{Lin}_{\prec}^+
(\mathbb{S}_{d_0})$ and $\tmop{Lin}_{\prec}^{\mathcal{S}} (\mathbb{S})$ for
the set of contracting strongly linear maps $\mathbb{S}_{d_0} \longrightarrow
\mathbb{S}_{d_0}$ and the set of contracting regular maps $\mathbb{S}
\longrightarrow \mathbb{S}$ respectively.

Lastly, we write $1$-$\tmop{Aut}^{\mathcal{S}} (\mathbb{S})$ for the group,
under composition, of regular algebra automorphisms $\sigma$ of $\mathbb{S}$
such that $\sigma (f) \sim f$ for all $f \in \mathbb{S}$. Since
each $C \tmop{Id}_{\mathbb{S}_{d_0}} + \tmop{Lin}^+_{\prec}
(\mathbb{S}_{d_0})$ has evaluations in the sense of
{\cite[Definition~2.3]{BKKMP:strong}}, and in view of the definition of
regularity, so has $C \tmop{Id}_{\mathbb{S}} +
\tmop{Lin}^{\mathcal{S}}_{\prec} (\mathbb{S})$. So
{\cite[Theorem~2.14]{BKKMP:strong}} applies and yields:

\begin{proposition}
  The set $\tmop{Der}^{\mathcal{S}}_{\prec} (\mathbb{S})$ of contracting
  regular derivations on $\mathbb{S}$ is a group for the
  Baker-Campbell-Hausdorff operation
  \begin{equation}
    \partial \ast \mathd \assign \partial + \mathd + \frac{1}{2}  \llbracket
    \partial, \mathd \rrbracket + \frac{1}{12}  (\llbracket \partial,
    \llbracket \partial, \mathd \rrbracket \rrbracket - \llbracket \mathd,
    \llbracket \partial, \mathd \rrbracket \rrbracket) + \cdots \label{eq-BCH}
    .
  \end{equation}
  Furthermore we have a group isomorphism
  \begin{eqnarray*}
    \exp \of (\tmop{Der}^{\mathcal{S}}_{\prec} (\mathbb{S}),\ast) & \longrightarrow &
    (1 \text{-} \tmop{Aut}^{\mathcal{S}} (\mathbb{S}),\circ)\\
    \partial & \longmapsto & \tmop{Id}_{\mathbb{S}} + \partial + \frac{1}{2}
    \partial \circ \partial + \frac{1}{6} \partial \circ \partial \circ
    \partial + \cdots .
  \end{eqnarray*}
\end{proposition}

Hidden terms in (\ref{eq-BCH}) are $\mathbb{Q}$-linear combinations of iterated
Lie brackets of lengths $> 3$.

\begin{remark}
  \label{rem-valuation-sum}We will freely use the fact that if a family
  $(f_i)_{i \in I} \in \mathbb{S}^I$ is summable and $f_i \preccurlyeq g$ for
  all $i \in I$, then $\sum_{i \in I} f_i \preccurlyeq g$. This follows from
  the fact that $\tmop{supp} \sum_{i \in I} f_i \subseteq \bigcup_{i \in I}
  \tmop{supp} f_i$, see {\cite[Section~3]{vdH:noeth}}.
\end{remark}

\subsection{Contractive hull of a derivation}

We now introduce a slight generalisation of the class of groups defined in
{\cite[Section~7.2]{Bag:c-val}}. Let $\partial \of \mathbb{S} \longrightarrow
\mathbb{S}$ be a fixed regular derivation with kernel $\tmop{Ker} (\partial) =
C$. The {\tmem{contractive hull}} of $\partial$ is defined as the following
subset of $\mathbb{S}$:
\[ \tmop{Cont} (\partial) \assign \left\{ f \in \mathbb{S} \suchthat f
   \partial \text{ is contracting} \right\} . \]
Identifying each $f \in \tmop{Cont} (\partial)$ with the regular contracting
derivation $f \partial$, we obtain a Lie bracket $\llbracket \cdot, \cdot
\rrbracket : \tmop{Cont} (\partial) \times \tmop{Cont} (\partial)
\longrightarrow \tmop{Cont} (\partial) ; (f, g) \mapsto f \partial (g) -
\partial (f) g$ on $\tmop{Cont} (\partial)$. It is easy to see that
$\tmop{Cont} (\partial) \partial$ is closed under sums of summable families.
Thus $\tmop{Cont} (\partial)$ is a group for the operation
\begin{equation}
  f \ast g \assign f + g + \frac{1}{2}  \llbracket f, g \rrbracket +
  \frac{1}{12}  (\llbracket f, \llbracket f, g \rrbracket \rrbracket -
  \llbracket g, \llbracket f, g \rrbracket \rrbracket) + \cdots
  \label{eq-BHC-cont} .
\end{equation}
The inverse of an $f \in \tmop{Cont} (\partial)$ for $\ast$ is simply $- f$.
We also have the following consequences of {\cite[Lemmas~7.14 and~7.15 and
Remark~7.19]{Bag:c-val}}. We give the proofs here for completion.

\begin{lemma}
  \label{lem-Lie-cont}For $f, g \in \tmop{Cont} (\partial) \backslash \{ 0
  \}$, we have $\llbracket f, g \rrbracket \prec f, g$.
\end{lemma}

\begin{proof}
  We may switch $f$ and $g$, so it suffices to show that $\llbracket f, g
  \rrbracket \prec f$. Since $g \partial$ is contracting, we have $g \partial
  (g) \prec g$, which means that $\partial (g) \prec 1$. We deduce that $f
  \partial (g) \prec f$. We also have $\partial (f) g = (g \partial) (f) \prec
  f$ since $g \partial$ is contracting. We deduce that $\llbracket f, g
  \rrbracket \prec f$.
\end{proof}

\begin{lemma}
  \label{lem-minus}For $f, g \in \tmop{Cont} (\partial) \setminus \{ 0 \}$
  with $f \neq g$, we have $f - g \succ \llbracket f, g \rrbracket$.
\end{lemma}

\begin{proof} We have $\llbracket f, g \rrbracket = \llbracket f-g, g \rrbracket$, so this follows from Lemma~\ref{lem-Lie-cont}.
\end{proof}

In view of (\ref{eq-BHC-cont}), Lemmas~\ref{lem-Lie-cont} and~\ref{lem-minus}
and \Cref{rem-valuation-sum}, we obtain:

\begin{corollary}
  \label{cor-sum-BCH}For $f, g \in \tmop{Cont} (\partial) \backslash \{ 0 \}$,
  we have $f \ast g \sim f + g$.
\end{corollary}

\begin{lemma}
  \label{prop-BCH-nearly-Abelian}\label{lem-comm12}For all $f, g \in
  \tmop{Cont} (\partial)$, we have
  \begin{eqnarray}
    f \ast g \ast (- f) & \sim & g \text{\quad and}  \label{eq-lem-comm12-1}\\
    f \ast g \ast (- f) - g & \sim & \llbracket f, g \rrbracket . 
    \label{eq-lem-comm12-2}
  \end{eqnarray}
\end{lemma}

\begin{proof}
  If $f = 0$ or $g = 0$, then $f \ast g \ast (- f) = g$. Suppose that $f, g
  \neq 0$. If $\llbracket f, g \rrbracket = 0$, then $f \ast g = f + g = g + f
  = g \ast f$ so $f \ast g \ast (- f) = g$. Suppose that $\llbracket f, g
  \rrbracket \neq 0$ and set $A \assign \llbracket f, g \rrbracket$. So $A
  \prec g$ by \Cref{lem-Lie-cont}. We see with \Cref{lem-Lie-cont} that $f
  \ast g = f + g + \frac{1}{2} A + \varepsilon$ for an $\varepsilon \prec A$.
  So
  \begin{eqnarray*}
    f \ast g \ast (- f) & = & \left( f + g + \frac{1}{2} A + \varepsilon
    \right) \ast (- f)\\
    & = & g + \frac{1}{2} A + \varepsilon + \frac{1}{2}  \left( \llbracket f,
    - f \rrbracket + \llbracket g, - f \rrbracket + \left\llbracket
    \frac{1}{2} A + \varepsilon, - f \right\rrbracket \right) + \cdots\\
    & = & g + \frac{1}{2} A + \varepsilon + \frac{1}{2}  \left( \llbracket g,
    - f \rrbracket + \left\llbracket \frac{1}{2} A + \varepsilon, - f
    \right\rrbracket \right) + \cdots\\
    & = & g + A + B
  \end{eqnarray*}
  for $B = \varepsilon + \frac{1}{2}  \left\llbracket \frac{1}{2} A +
  \varepsilon, - f \right\rrbracket + \cdots \prec A$ by \Cref{lem-Lie-cont}.
  This shows that $f \ast g \ast (- f) \sim g$ and that $f \ast g \ast (- f) -
  g \sim A = \llbracket f, g \rrbracket$.
\end{proof}

\begin{example}
  If $D = \{ \bullet \}$, $\Gamma_{\bullet} = \Gamma = (\mathbb{Z}, +, 0, <)$,
  and $\partial = \frac{\mathd}{\mathd t}$ is the derivation with respect to
  $t$ on $C ( \! (t) \! ) = C ( \! (\mathbb{Z}) \!
  )$, then $\tmop{Cont} (\partial) = \left\{ f \in C ( \!
  (\mathbb{Z}) \! ) \suchthat f \prec t^2 \right\}$.
\end{example}

\subsection{Integration and asymptotic integration}

From now on, given $f \in \mathbb{S}$ and $g \in \mathbb{S}^{\times}$, we
sometimes write $f' \assign \partial (f)$ and $g^{\dag} \assign \partial (g) /
g$. {\tmstrong{We make the assumption that $(\mathbb{S}, \preccurlyeq,
\partial)$ is an H-asymptotic field}} in the sense of {\cite[p 324]{vdH:mt}}.
In other words, we assume that for all $f, g \in \mathbb{S}^{\times}$ with $f,
g \prec 1$, we have
\[ f \prec g \Longleftrightarrow f' \prec g'
   \text{{\hspace{3em}}and{\hspace{3em}}$f \prec g \Longrightarrow f^{\dag}
   \succcurlyeq g^{\dag} .$} \]
Since $\tmop{Ker} (\partial) +\mathbb{S}^{\prec}$ is the valuation ring of
$(\mathbb{S}, v)$, this means that $v$ is a differential valuation on
$(\mathbb{S}, \partial)$ in the sense of {\cite[Definition, p 4]{Rosli80}}. We
then have well-defined maps
\[ \mathord{'} \of \Gamma \setminus \{ 0 \} \longrightarrow \Gamma \: ; \: v
   (g) \mapsto v (g')
   \text{{\hspace{3em}}and{\hspace{3em}}$\mathord{\:^{\dag}} \of \Gamma
   \setminus \{ 0 \} \longrightarrow \Gamma \: ; \: v (g) \mapsto v
   (g^{\dag})$}, \]
and the structure $\left( \Gamma, +, 0, <, \mathord{\:^{\dag}} \right)$ is
called the {\tmem{asymptotic couple}} of~$(\mathbb{S}, \preccurlyeq,
\partial)$ (see {\cite[p 325]{vdH:mt}}). By {\cite[Theorem~4]{Rosli80}}, it is
an asymptotic couple in the sense of {\cite[p 273]{vdH:mt}}. We have:

\begin{lemma}
  \label{lem-monotone}{\tmem{{\cite[Lemma~6.5.4]{vdH:mt}}}} The map
  $\mathord{'} \of \Gamma \setminus \{ 0 \} \longrightarrow \Gamma$ is
  strictly increasing.
\end{lemma}

Given $f \in \mathbb{S}$, an asymptotic integral of $f$ in $(\mathbb{S},
\preccurlyeq, \partial)$ is an element $A \in \mathbb{S}$ with $A' \sim f$.
Such an element may not exist. If such an element always exist, then
$(\mathbb{S}, \preccurlyeq, \partial)$ is said to be {\tmem{closed under
asymptotic integration}}. In the case when $D$ is a singleton, being closed
under asymptotic integration is equivalent {\cite[Lemma~1.7]{vdH:dagap}} to
the existence of a strongly linear right inverse for $\partial$.

A crucial property of $(\mathbb{S}, \preccurlyeq, \partial)$ is that
{\cite[Theorem~9.2.1]{vdH:mt}} there is at most one $\gamma \in \Gamma$ such
that $\gamma \nin (\Gamma \setminus \{ 0 \})'$. If such an element $\gamma$
exist, then we call it the {\tmem{pseudo-gap}} of $(\mathbb{S}, \preccurlyeq,
\partial)$. If no such $\gamma$ exists,
then $\mathbb{S}$ is closed under asymptotic integration. Indeed, since
$\tmop{Ker} (\partial) = C$, for any $f \in \mathbb{S} \setminus \{ 0 \}$ and
any $g$ with $v (g') = v (f)$, the element $cg$ is an asymptotic integral of
$f$ where $c$ is the leading coefficient of $f / g'$. 

An important subset of $\Gamma$ is the psi-set
\[ \Psi \assign \{ \gamma^{\dag} \suchthat \gamma \in \Gamma \wedge \gamma > 0
   \} \subseteq \Gamma . \]
Indeed, the pseudo-gap of $(\mathbb{S}, \preccurlyeq, \partial)$ is either the
maximum of $\Psi$ if this maximum exists, or the unique $\gamma \in \Gamma$
with $\Psi < \gamma < \{ \gamma' \suchthat \gamma \in \Gamma \wedge \gamma > 0
\}$ if $\Psi$ has no maximum ({\cite[Theorem~9.2.1
and~Corollary~9.2.4]{vdH:mt}}) and such an element exists. We will require the following technical lemma:

\begin{lemma}\label{lem-careful-gaps}
    Let $d_0,d \in D$ with $d_0\leqslant d$ and let $\gamma \in \Gamma_d$ such that $\gamma'$ is the pseudo-gap of $(\mathbb{S}_{d_0},\preccurlyeq,\partial)$. Then we have $\mathbb{N} \gamma + \Gamma_{d_0} \subseteq (\mathbb{N} \gamma + \Gamma_{d_0})'$.
\end{lemma}

\begin{proof}
    Let $\alpha \in \Gamma_{d_0}$ and $n \in \mathbb{N}$. If $\alpha \notin (\Gamma_{d_0})'$, then $\alpha = \gamma'$. Since $n+1\neq 0$, we have $n \gamma +\alpha = ((n+1) \gamma)' \in (\mathbb{N} \gamma)'$. Suppose now that $\alpha = \beta'$ for a $\beta \in \Gamma_{d_0}$. If $n =0$ then $n \gamma + \alpha = \beta' \in \Gamma_{d_0}'$. We now assume that $n>0$. We have $\gamma'\geqslant (\Gamma_{d_0}^{\neq})^{\dag}$ and $\gamma <0$, so $\gamma^{\dag} =\gamma'-\gamma>(\Gamma_{d_0}^{\neq})^{\dag}$. In particular $\gamma^{\dag} > \beta^{\dag}$, whence $(n\gamma)^{\dag} = \gamma^{\dag}>\beta^{\dag}$. Thus $(n\gamma)'+\beta>n\gamma+\beta' = n \gamma+\alpha$, which means in view of the Leibniz rule that $(n\gamma + \beta)' = n \gamma+\alpha$, so $n \gamma+\alpha \in \Gamma_d'$.\end{proof}

\begin{lemma}
  \label{lem-cont-cond}For $f \in \mathbb{S}$, we have $f \in \tmop{Cont}
  (\partial)$ if and only if $v (f) + \Psi > 0$.
\end{lemma}

\begin{proof}
  We have $f \in \tmop{Cont} (\partial)$ if and only if $fg' \prec g$ for all
  $g \in \mathbb{S}$, i.e. if and only if $f \prec \frac{1}{g^{\dag}}$ for all
  $g \in \mathbb{S} \setminus C$. Since $(g^{- 1})^{\dag} = - g^{\dag} \asymp
  g^{\dag}$ for all $g \in \mathbb{S}$ and since $(c + \varepsilon)^{\dag}
  \asymp \varepsilon' \prec \varepsilon^{\dag}$ for all $c \in C^{\times}$ and
  $\varepsilon \in \mathbb{S}$ with $\varepsilon \prec 1$, it is equivalent
  that $fg^{\dag} \prec 1$ for all $g \in \mathbb{S}^{\prec}$, hence the
  result.
\end{proof}

\begin{remark}
  \label{rem-extension-of-scalars}Given a field extension $L / C$, there is a
  natural ``strong extension of scalars'' $\mathbb{S} \otimes^+_C L$ given as
  the direct limit of the system of fields $\left( L \left( \! (\Gamma_d)
  \! \right) \right)_{d \in D}$. The map $\partial$ extends uniquely into a
  regular derivation $\partial_L \of \mathbb{S} \otimes^+_C L \longrightarrow
  \mathbb{S} \otimes^+_C L$, and this extension preserves all relevant
  properties of $(\mathbb{S}, \preccurlyeq, \partial)$. Namely $(\mathbb{S}
  \otimes^+_C L, \preccurlyeq, \partial_L)$ is an H-asymptotic field with the
  same asymptotic couple as $(\mathbb{S}, \preccurlyeq, \partial)$. So our
  results apply without change to $\mathbb{S} \otimes^+_C L$.
\end{remark}

For $d \in D$, we write $\mathfrak{M}_d$ for the subset of series in
$\mathbb{S}_d$ whose support is a singleton $\{ \gamma \}, \gamma \in \Gamma$
and whose value at $\gamma$ is $1$. So $\mathfrak{M}_d$ is a subgroup of
$\mathbb{S}_d^{\times}$ and $v \of (\mathfrak{M}_d, \cdot, 1, \succ)
\longrightarrow (\Gamma_d, +, 0, <)$ is an isomorphism. Elements in
$\mathfrak{M}= \bigcup_{d \in D} \mathfrak{M}_d$ are called
{\tmem{monomials}}, and elements in $C\mathfrak{M} \subseteq \mathbb{S}$ are
called {\tmem{terms}}. Given $f \in \mathbb{S}^{\times}$, there is a unique
term $\tmop{lead} (f)$ called the {\tmem{leading term}} of $f$ such that $f
\sim \tmop{lead} (f)$. 

\begin{remark}\label{rem-ast-supp}
    If $f,g \in \mathbb{S}$ and $(f_i)_{i \in I}$ is a summable family in $\mathbb{S}$, then $\tmop{supp} fg$ is contained in $\tmop{supp} f+\tmop{supp} g$, and $\tmop{supp} \sum \limits_{i \in I} f_i$ is contained in $\bigcup \limits_{i \in I} \tmop{supp} f_i$ (see \cite[Section~3]{vdH:noeth}). It follows that for $f,g \in \tmop{Cont}(\partial)$ and for any submonoid $M \subseteq \Gamma$ with $\tmop{supp} f \cup \tmop{supp} g \subseteq M$ and $\partial(C(\!(M)\!) )\subseteq C(\!(M)\!)$, the support of $f\ast g$ is contained in $M$.
\end{remark}

We now define our notion of compatibility between asymptotic integration and the directed system. We will see that it is satisfied in most natural fields of generalised power series, such as transseries.

\begin{definition}
  \label{def-semi-regular}We say that {\tmem{{\tmstrong{asymptotic integration
  is regular on}}}} $\mathbb{S}$ if for all $d_0 \in D$, there is a $d \in D$
  with $d \geqslant d_0$ such that for each $f \in \mathbb{S}_{d_0}$ whose
  valuation is not the pseudo-gap of $(\mathbb{S}, \preccurlyeq, \partial)$,
  there is a term $\tau \in C\mathfrak{M}_d$ with $\tau' \sim f$ and $\tau' \in
  \mathbb{S}_{d_0}$. We then say that such a $d$ is associated to $d_0$. We say that $\mathbb{S}$ {\tmem{{\tmstrong{has regular
  asymptotic integration}}}} if it is closed under asymptotic integration and
  asymptotic integration is regular on $\mathbb{S}$.
\end{definition}

\section{Conjugacy of derivations}\label{section-conjugacy}

We first consider the following approximation of the conjugacy problem: given
$f, g \in \tmop{Cont} (\partial) \setminus \{ 0 \}$ with $f \neq g$, when is
there a $\varphi \in \tmop{Cont} (\partial)$ such that $\varphi \ast g \ast (-
\varphi) - f \prec f - g$? We say that a $\varphi \in \tmop{Cont} (\partial)$
satisfying this is an {\tmem{asymptotic conjugating element}} for $(f, g)$,
and we say that $f$ and $g$ are asymptotically conjugate if such an element
exists.

\begin{lemma}
  \label{lem-approximate-conjugacy}Let $f, g \in \tmop{Cont} (\partial)
  \setminus \{ 0 \}$ with $f \neq g$. Then $f$ and $g$ are asymptotically
  conjugate if and only if $f \sim g$ and $\frac{f - g}{g^2}$ has an
  asymptotic integral $A$ in $\mathbb{S}$ with $v (A) + v (g) + \Psi > 0$,
  their asymptotic conjugating elements are exactly the series $gA$ for such
  $A$.
\end{lemma}

\begin{proof}
  Let $d \in D$ with $f, g \in \mathbb{S}_d$. In view of
  \Cref{eq-lem-comm12-1}, a first necessary condition is that $f \sim g$.
  Suppose that $f \sim g$ and set $\delta \assign f - g$. We want to find a $y
  \in \tmop{Cont} (\partial)$ such that $y \ast g \ast (- y) - g - \delta
  \prec \delta$. Let $y \in \tmop{Cont} (\partial)$. By
  \Cref{eq-lem-comm12-2}, we have $y \ast g \ast (- y) - g \sim \llbracket y,
  g \rrbracket$. Thus $y$ is an asymptotic conjugating element for $(f, g)$ if
  and only if
  \[ y' - g^{\dag} y \sim \frac{- \delta}{g} . \]
  The solutions are of the form $gA$ where $A' \sim \frac{- \delta}{g^2}$.
  Thus $f$ and $g$ are asymptotically conjugate if and only if $- \delta /
  g^2$ has an asymptotic integral $A$ in $\mathbb{S}$ such that $gA$ lies in
  $\tmop{Cont} (\partial)$. We conclude with \Cref{lem-cont-cond}.
\end{proof}

\begin{lemma}
  \label{lem-exact}Suppose that asymptotic integration is regular on
  $\mathbb{S}$. Let $f, g \in \tmop{Cont} (\partial) \setminus \{ 0 \}$ be
  asymptotically conjugate. There are a $d \in D$, an ordinal $\lambda > 0$
  and a strictly $\prec$-decreasing sequence $(\tau_{\gamma})_{\gamma <
  \lambda}$ of terms in $\tmop{Cont} (\partial) \cap \mathbb{S}_d$ such that
  writing $\varphi_{\eta} \assign \sum_{\gamma < \eta} \tau_{\gamma}$ for all
  $\eta \leqslant \lambda$, the sequence $(\varphi_{\eta} \ast g \ast (-
  \varphi_{\eta}) - f)_{\gamma \leqslant \lambda}$ is strictly
  $\prec$-decreasing, and one of the following occurs:
  \begin{enumeratealpha}
    \item $\varphi_{\lambda} \ast g \ast (- \varphi_{\lambda}) = f$.
    
    \item $v \left( \frac{f - \varphi_{\lambda} \ast g \ast (-
    \varphi_{\lambda})}{g^2} \right)$ is the pseudo-gap of $(\mathbb{S},
    \preccurlyeq, \partial)$.
  \end{enumeratealpha}
\end{lemma}

\begin{proof}
  Let $d_0 \in D$ such that $f, g \in \mathbb{S}_{d_0}$, and let $d \geqslant
  d_0$ be as in the definition of regular asymptotic integration, with respect
  to $d_0$. If there is a $\mathfrak{u} \in \mathfrak{M}_d$ such that $v(\mathfrak{u}')$ is the pseudo-gap of $(\mathbb{S}_{d_0},\preccurlyeq,\partial)$, and $\mathfrak{u}' \in \mathbb{S}_{d_0}$, then we set $\mathfrak{m}\assign \mathfrak{u}$. If not, we set $\mathfrak{m}\assign 1$. By induction on an ordinal $\alpha$, we construct a strictly
  $\prec$-decreasing sequence $(\tau_{\gamma})_{\gamma < \alpha}$ of
  terms in $\tmop{Cont} (\partial) \cap (C(\! (\Gamma_{d_0} +\mathbb{N} v(\mathfrak{m}))\!) )$ such that defining $\varphi_{\eta},\eta\leqslant \alpha$ as above, the sequence
  $(\varphi_{\eta} \ast g \ast (- \varphi_{\eta}) - f)_{\eta \leqslant
  \alpha}$ is strictly $\prec$-decreasing and that for all $\gamma < \alpha$
  and all $\varepsilon \prec \tau_{\gamma}$, we have $(\varphi_{\gamma + 1} +
  \varepsilon) \ast g \ast (- \varphi_{\gamma + 1} - \varepsilon) - f \prec
  \varphi_{\gamma} \ast g \ast (- \varphi_{\gamma}) - f$. The $\varepsilon$ here is a stand-in for $\varphi_{\alpha}-\varphi_{\eta}$ for various values of $\eta<\alpha$.

  Let $\alpha$ such
  that for all $\eta < \alpha$, the sequence $(\tau_{\gamma})_{\gamma < \eta}$
  is defined and satisfies the conditions. We first make a few observations. Let $\gamma\leqslant \alpha$ and write $g_{\gamma} \assign \varphi_{\gamma} \ast g \ast (-\varphi_{\gamma})$. We have $g\sim g_{\gamma}$ by (\ref{eq-lem-comm12-1}), so $\frac{f-g_{\gamma}}{g^2}$ and $\frac{f-g_{\gamma}}{g_{\gamma}^2}$ have the same symptotic integrals. Recall that $(f-g_{\rho})_{\rho\leqslant\alpha}$ is strictly $\prec$-decreasing. So given $\rho_1\leqslant\rho_2\leqslant \alpha$, if $\frac{f-g_{\rho_i}}{g^2}$ has an symptotic integral $A_i,i \in\{1,2\}$, then we have $v(A_1)+v(g_{\rho_1})+\Psi>0 \Longrightarrow v(A_2)+v(g_{\rho_2})+\Psi>0$. Thus, by Lemma~\ref{lem-approximate-conjugacy}, the series $f$ and $g_{\gamma}$ are asymptotically conjugate if and only if $v(\frac{f-g_{\gamma}}{g^2})$ is not the pseudo-gap of $(\mathbb{S},\preccurlyeq,\partial)$.
  
  Let us now give the inductive proof. If $\alpha$ is a limit, including the case when $\alpha=0$, then we need only check that we have $(\varphi_{\gamma+1}+\varepsilon) \ast g \ast (-\varphi_{\gamma+1}-\varepsilon) \prec g_{\alpha}-f$ for all $\gamma<\alpha$ and $\varepsilon\prec \tau_{\gamma}$. But this follows by induction, noting that $\varphi_{\eta} - \varphi_{\gamma + 1}
  \prec \tau_{\gamma}$ for all $\gamma < \eta < \alpha$. Suppose that $\alpha
  = \eta + 1$ is a successor ordinal. If $g_{\eta} = f$, then, setting $\lambda \assign \alpha$, we are done.
  Suppose that $g_{\eta} \neq f$. If $\mu
  \assign v \left( \frac{f - g_{\eta}}{g^2} \right)$ is the pseudo-gap of $(\mathbb{S},
  \preccurlyeq, \partial)$, then we set $\lambda \assign \alpha$ and we are
  done. 
  
  Suppose now that $\mu$ is not the pseudo-gap of $(\mathbb{S}, \preccurlyeq,
  \partial)$. Since $\mathfrak{m}' \in \mathbb{S}_{d_0}$, Remark~\ref{rem-ast-supp} applies, and the set $\tmop{Cont} (\partial) \cap (C(\!  (\Gamma_{d_0} +\mathbb{N} v(\mathfrak{m})) \!) )$ is closed under $\ast$. We see with Lemma~\ref{lem-careful-gaps} that $v(\frac{f - g_{\eta}}{g^2})$ is not a pseudo-gap in $(\mathbb{S}_d,\preccurlyeq,\partial)$, so by regularity of asymptotic integration, there is a term $\tau$ in $\mathbb{S}_d$ with $\tau' \in
  \mathbb{S}_d$ and $\tau' \sim - \frac{f - g_{\eta}}{g^2}$. Note that $\tau$ is unique, and in view of Lemma~\ref{lem-careful-gaps}, we have $\tau \in C^{\times}\mathfrak{M}_{d_0} \mathfrak{m}^{\mathbb{N}}$. So we may extend the sequence by setting $\tau_{\eta} \assign \tmop{lead} (g) \tau$, which by \Cref{lem-approximate-conjugacy} lies in $\tmop{Cont}(\partial)$. We have $\tau'\prec (\frac{\tau_{\gamma}}{\tmop{lead}(g)})'$ for all $\gamma<\eta$ since $(f-g_{\gamma})_{\gamma<\alpha}$ is strictly $\prec$-decreasing. We deduce with Lemma~\ref{lem-monotone} that $\tau\prec \frac{\tau_{\gamma}}{\tmop{lead}(g)}$, whence $\tau_{\eta}\prec \tau_{\gamma}$ for all $\gamma<\eta$.
  
  Thus $\varphi_{\alpha}$ is well-defined, and it remains to show that the conditions related to that sequence are satisfied. By
  \Cref{lem-approximate-conjugacy} with $(g_{\eta},f)$ in the role of $(f,g)$, we have
  \[ s \ast \varphi_{\eta} \ast g \ast (- \varphi_{\eta} \ast (- s)) - f
     \prec \varphi_{\eta} \ast g \ast (- \varphi_{\eta}) - f \]
  for all $s \in \mathbb{S}$ with $s \sim \tau_{\eta}$. We note with
  \Cref{cor-sum-BCH} that $\tau_{\eta} + \varphi_{\eta}$ is of the form $s
  \ast \varphi_{\eta}$ for $s \assign (\tau_{\eta} + \varphi_{\eta}) \ast (-
  \varphi_{\eta}) \sim \tau_{\eta}$, so we have
  \[ (\varepsilon + \tau_{\eta} + \varphi_{\eta}) \ast g \ast (-
     (\varphi_{\eta} + \tau_{\eta} + \varepsilon)) - f \prec \varphi_{\eta}
     \ast g \ast (- \varphi_{\eta}) - f \]
  for all $\varepsilon \in \mathbb{S}$ with $\varepsilon \prec \tau_{\eta}$,
  as claimed. The sequence $(\tau_{\gamma})_{\gamma < \alpha}$ is strictly
  $\prec$-decreasing, hence injective. Since $\mathbb{S}_d$ is a set, there is an ordinal $\lambda$ at which the
  process stops, i.e. one of the cases of the lemma occurs.
\end{proof}

Note that the recularity of asymptotic integration plays a crucial role in the previous proof in allowing us to remain in $\mathbb{S}_d$ within the induction, so that $\varphi_{\lambda}$ be well-defined.

\begin{proposition}
  \label{prop-with-Poincar�}Suppose that asymptotic integration is regular on
  $\mathbb{S}$. Let $f, g \in \tmop{Cont} (\partial) \setminus \{ 0 \}$ with
  $f \sim g$, and assume that $\mu \assign v \left( \frac{f - g}{g^2} \right)$
  satisfies $\mu > (- v (g) - \Psi)'$ and lies above any pseudo-gap of
  $(\mathbb{S}, \preccurlyeq, \partial)$. Then $f$ and $g$ are conjugate in
  $\tmop{Cont} (\partial)$.
\end{proposition}

\begin{proof}
  We may assume that $f \neq g$. Since $\mu$ lies above any pseudo-gap of
  $(\mathbb{S}, \preccurlyeq, \partial)$, there is an $\alpha \in \Gamma
  \setminus \{ 0 \}$ with~$\alpha' = \mu$. By \Cref{lem-monotone}, we have
  $\alpha + v (g) + \Psi > 0$, so $f$ and $g$ are asymptotically conjugate by
  \Cref{lem-approximate-conjugacy}. We thus have a sequence
  $(\varphi_{\eta})_{\eta \leqslant \lambda}$ as in \Cref{lem-exact} for $(f,
  g)$. Since the sequence $(f - \varphi_{\eta} \ast g \ast (-
  \varphi_{\eta}))_{\eta \leqslant \lambda}$ is strictly
  $\preccurlyeq$-decreasing , we have $v \left( \frac{f - \varphi_{\eta} \ast
  g \ast (- \varphi_{\eta})}{g^2} \right) > \mu$ for all $\eta \leqslant
  \lambda$, so the second case of \Cref{lem-exact} cannot occur. Therefore the
  first one does, i.e. $f$ and $g$ are conjugate.
\end{proof}

\begin{theorem}
  \label{th-main}Suppose that $(\mathbb{S}, \preccurlyeq, \partial)$ has
  regular asymptotic integration. Let $f, g \in \tmop{Cont} (\partial)
  \setminus \{ 0 \}$. Then $f$ and $g$ are conjugate in $\tmop{Cont}
  (\partial)$ if and only if $f \sim g$ and $\mu \assign v \left( \frac{f -
  g}{g^2} \right)$ satisfies $\mu > (- v (g) - \Psi)'$.
\end{theorem}

\begin{proof}
  If $f$ and $g$ are conjugate, then they are asymptotically conjugate, so by
  \Cref{lem-approximate-conjugacy}, there is an $\alpha \in \Gamma \setminus
  \{ 0 \}$ with $\alpha' = \mu$ and $\alpha + v (g) + \Psi > 0$. We deduce by
  \Cref{lem-monotone} that $\alpha' > (- v (g) - \Psi)'$. Conversely, suppose
  that $\mu > (- v (g) - \Psi)'$ and that $f \sim g$. By asymptotic
  integration, we find an $\alpha \in \Gamma$ with $\alpha' = \mu$, whence
  $\alpha + v (g) + \Psi > 0$ by \Cref{lem-monotone}. Since $(\mathbb{S},
  \preccurlyeq, \partial)$ has no pseudo-gap, we conclude with
  \Cref{prop-with-Poincar�}.
\end{proof}

\begin{corollary}
  \label{cor-initial}Suppose that $(\mathbb{S}, \preccurlyeq, \partial)$ has
  regular asymptotic integration. Then given $f \in \tmop{Cont} (\partial)$,
  the set of series $\varepsilon \in \mathbb{S}$ such that $f + \varepsilon$
  is a conjugate of $f$ in $\tmop{Cont} (\partial)$ is downward closed for
  $\preccurlyeq$.
\end{corollary}

We will see in \Cref{subsection-resonnance} that the conclusion of
\Cref{cor-initial} does not hold in the presence of pseudo-gaps. We conclude
with a more standard formulation of the conjugacy problem for derivations.

\begin{theorem}
  \label{th-main-standard}Suppose that $(\mathbb{S}, \preccurlyeq, \partial)$
  has regular asymptotic integration. For $f, g \in \tmop{Cont} (\partial)
  \setminus \{ 0 \}$ such that $f \sim g$ and $v \left( \frac{f - g}{g^2}
  \right) > (- v (g) - \Psi)'$, the derivations $f \partial$ and $g \partial$
  are conjugate over $\tmop{Aut} (\mathbb{S})$, i.e. there is a $\sigma = \exp
  (h \partial) \in \tmop{Aut} (\mathbb{S})$ with $\sigma \circ (g \partial)
  \circ \sigma^{\tmop{inv}} = f \partial$.
\end{theorem}

\begin{proof}
  By \Cref{th-main}, there is an $h \in \tmop{Cont} (\partial)$ with $(- h)
  \ast g \ast h = f$. Set $\sigma \assign \exp (h \partial)$. Then $\sigma
  \circ \exp (g \partial) \circ \sigma^{\tmop{inv}} = \exp (f \partial)$. Now
  the conjugation by $\sigma$ is a strongly linear algebra automorphism of
  $\tmop{Der}^{\mathcal{S}}_{\prec} (\mathbb{S})$ (see
  {\cite[Proposition~1.28]{BKKMP:strong}}), while $\log$ is given as a power series of its argument with constant coefficients. So $f \partial = \log (\sigma
  \circ \exp (g \partial) \circ \sigma^{\tmop{inv}}) = \sigma \circ \log (\exp
  (g \partial)) \circ \sigma^{\tmop{inv}} = \sigma \circ (g \partial) \circ
  \sigma^{\tmop{inv}}$.
\end{proof}

\section{Conjugacy of formal series}

We now state the results in \Cref{section-conjugacy} in terms of conjugacy of
series under composition.

\subsection{Transseries}

Let $\mathbb{T}$ denote the ordered field of logarithmic-exponential
transseries {\cite{Ec92,vdDMM01}} together with its standard derivation
$\partial \of f \mapsto f'$ and let $\mathfrak{M}$ denote the set of
transmonomials in $\mathbb{T}$, i.e. $\mathfrak{M}$ is a specific section of
the valuation $v \of \mathbb{T} \longrightarrow \Gamma \cup \{ \infty \}$. We
identify each value $\gamma \in v (\mathbb{T}^{\times})$ with the
corresponding transmonomial $\mathfrak{m} \in \mathfrak{M}$ with $v
(\mathfrak{m}) = \gamma$. Recall that $\mathbb{T}$ is a direct limit of Hahn
fields $\mathbb{T}_{m, n}, m, n \in \mathbb{N}$ where $\mathbb{T}_{m, n}
=\mathbb{R} \left( \! (\mathfrak{M}_{m, n}) \! \right)$ is the field of
transseries with exponential depth $\leqslant m$ and logarithmic depth
$\leqslant n$ (see {\cite{Ec92,Edgar:begin}}). The directed system $(\mathfrak{M}_{m,n})_{m,n \in \mathbb{N}}$ has inclusions $\mathfrak{M}_{m,n} \longrightarrow \mathfrak{M}_{m',n'}$ whenever $n'\geqslant n$ and $m'\geqslant m+(n'-n)$ (see \cite[(2.7)]{vdDMM01}). The derivation $\partial$ is
regular {\cite[Section~3]{vdDMM01}}. The identity series is denoted $x$, and
we have $x' = 1$, so $x^{\dag} = x^{-1}$. Note that contrary to \cite{PeResRoSer:linear,PeResRoSer:normal,Peran:parabolic,Peran:hyperbolic}, we take $x$ to be positive infinite as opposed to positive infinitesimal. There is a bijective morphism $\log \of (\mathbb{T}^{>},
\cdot, 1, <) \longrightarrow (\mathbb{T}, +, 0, <)$ where $\mathbb{T}^{>} = \{
f \in \mathbb{T} \suchthat f > 0 \}$. Given $n \in \mathbb{N}$, we denote
by $\log_n x$ the $n$-th iterate of $\log$ applied at~$x$. We write $\mathbb{T}_{\log} \assign \bigcup \limits_{ n \in \mathbb{N}} \mathbb{T}_{0,n}$ for the field of logarithmic transseries {\cite{Gre:phd}}.

Set
$\mathbb{T}^{>\mathbb{R}} \assign \{ f \in \mathbb{T} \suchthat f >\mathbb{R}
\}$. We recall that $\mathbb{T}$ is equipped with a formal composition law
$\mathord{\circ} \of \mathbb{T} \times \mathbb{T}^{>\mathbb{R}}
\longrightarrow \mathbb{T}$ (see {\cite[Section~6]{vdDMM01}}). We write
$\check{\circ}$ for the inverse law on $\mathbb{T}^{>\mathbb{R}}$, given by $f \:
\check{\circ} \: g \assign g \circ f$ for all $f, g \in
\mathbb{T}^{>\mathbb{R}}$.

Let $\mathbb{S}$ denote the subset of $\mathbb{T}$ of flat transseries, i.e.
series $f \in \mathbb{T}$ with $\mathfrak{m}^{\dag} \preccurlyeq x^{- 1}$ for
all $\mathfrak{m} \in \tmop{supp} f$. So $\mathbb{S}= \{ f \in \mathbb{T}
\suchthat \tmop{supp} f \subseteq \mathfrak{M}_{\precpreceq x} \}$ where
$\mathfrak{M}_{\precpreceq x} = \{ \mathfrak{m} \in \mathfrak{M} \suchthat
\mathfrak{m}^{\dag} \preccurlyeq x^{- 1} \}$. Note that
$\mathfrak{M}_{\precpreceq}$ is a subgroup of $\mathfrak{M}$, so $\mathbb{S}$
is a subfield of $\mathbb{T}$. We write $\mathbb{S}^{>\mathbb{R}} \assign \mathbb{S}
\cap \mathbb{T}^{>\mathbb{R}}$ and $\mathbb{S}^> \assign \mathbb{S}
\cap \mathbb{T}^>$.

\begin{example}
    The transmonomials $(\log_2 x)^{-5}$, $x^2$ and $\operatorname{e}^{\sqrt{\log x}}$ are flat, whereas $\operatorname{e}^{(\log x)^2}$ and $\operatorname{e}^x$ are not.
\end{example}

\begin{lemma}
    $\mathbb{S}$ is a differential subfield of $\mathbb{T}$
\end{lemma}

\begin{proof}
    It suffices to show that $\partial (\mathfrak{M}_{\precpreceq
x}) \subseteq \mathbb{S}$. We prove by induction on $m \in \mathbb{N}$ with $\mathfrak{m} \in \bigcup \limits_{n \in \mathbb{N}} \mathbb{T}_{m,n}$ that $\mathfrak{m} \in \mathfrak{M}_{\precpreceq
x} \Rightarrow \mathfrak{m}' \in \mathbb{S}$. This is clear if $m=0$ since $\mathbb{T}_{\log}' \subseteq \mathbb{T}_{\log}$. If $\mathfrak{m} \in \mathbb{T}_{m+1,n} \cap \mathfrak{M}_{\precpreceq
x} $ and the result holds at $m$, then we have $\tmop{supp} \mathfrak{m}' \subseteq \mathfrak{m} \tmop{supp} \log(\mathfrak{m})'$ where $\log(\mathfrak{m}) \in \mathbb{T}_{m,n+1}$. We conclude by induction.\end{proof}

\begin{lemma}
    $\mathbb{S}$ is closed under composition.
\end{lemma}

\begin{proof}
    Let $f \in \mathbb{S}$ with $f>\mathbb{R}$. It suffices to show that $\mathfrak{M}_{\precpreceq
x} \circ f \subseteq \mathbb{S}$. We prove by induction on $m \in \mathbb{N}$ with $\mathfrak{m} \in \bigcup \limits_{n \in \mathbb{N}} \mathbb{T}_{m,n}$ that $\mathfrak{m} \in \mathfrak{M}_{\precpreceq
x} \Rightarrow \mathfrak{m} \circ f \in \mathbb{S}$. The group $\mathfrak{M}_{\precpreceq
x}$ is closed under taking real powers, whence by \cite[6.8]{vdDMM01} so is $\mathbb{S}$. Moreover $\mathbb{S}$ is closed under $\log$. So the result holds if $m=0$. Suppose that $\mathfrak{m} \in \mathbb{T}_{m+1,n} \cap \mathfrak{M}_{\precpreceq
x}$ and that the result holds at $m$. Considering $\mathfrak{m}^{-1}$ if necessary, we may assume that $\mathfrak{m}$ is infinite. We have
\[\mathfrak{m} \circ f = \exp(\log(\mathfrak{m}) \circ f).\]

By construction of $\exp$ on $\mathbb{T}$ (see \cite[1.5--2.7]{vdDMM01}), each element $\mathfrak{n}$ in $\tmop{supp} \mathfrak{m} \circ f$ has the form $\mathfrak{n} = \mathfrak{u} \mathfrak{n}_1 \cdots \mathfrak{n}_k$ where $\mathfrak{u}$ is the largest monomial of $\exp(\log(\mathfrak{m}) \circ f)$, $k \in \mathbb{N}$, and $\mathfrak{n}_1,\dots,\mathfrak{n}_k$ are in $\tmop{supp} \log(\mathfrak{m}) \circ f$. As $\log \mathfrak{m} \in \mathbb{T}_{m,n+1}$, the induction hypothesis gives $\mathfrak{n}_1 \cdots \mathfrak{n}_k \in \mathbb{S}$. Lastly, note that $\mathfrak{u}^{\dag} \asymp (\mathfrak{m} \circ f)^{\dag}$. The chain rule \cite[Proposition~6.3]{vdDMM01} gives $(\mathfrak{m} \circ f)^{\dag} = f' \mathfrak{m}^{\dag} \circ f$. But $\mathfrak{m}^{\dag} \preccurlyeq x^{-1}$ and the composotion by $f$ on the right is valuation preserving, so $(\mathfrak{m} \circ f)^{\dag} \preccurlyeq f^{\dag} \preccurlyeq x^{-1}$. So $\mathfrak{u} \in \mathbb{S}$, whence $\mathfrak{m} \circ f \in \mathbb{S}$.
 We conclude by induction.
\end{proof}

We have a directed system $ \mathcal{S}\assign(\mathfrak{M}_{m,n} \cap \mathfrak{M}_{\precpreceq x})_{m,n \in \mathbb{N}}$ and we write $\mathbb{S}_{m,n} = \mathbb{R}(\! (\mathfrak{M}_{m,n} \cap \mathfrak{M}_{\precpreceq x})\!) = \mathbb{S} \cap \mathbb{T}_{m,n}$ for all $m,n \in \mathbb{N}$. We now focus on the directed system $\mathcal{S}$, H-field $(\mathbb{S}, \preccurlyeq, \partial)$, and the group
$(\tmop{Cont} (\partial), \ast, 0)$.

\begin{lemma}
  \label{lem-semi-regular}The direct limits $\mathbb{S}$, $\mathbb{T}$ and $\mathbb{T}_{\log}$ have regular asymptotic
  integration.
\end{lemma}

\begin{proof}
We prove this for $\mathbb{S}$, the proofs in the other cases being identical. Let $(m, n) \in \mathbb{N}^2$ and $f \in \mathbb{S}_{m, n}$. The pseudo-gap
  of $\mathbb{S}_{m, n}$ is $\gamma_n \assign v ((\log_{n + 1} x)')$ where
  $\log_{n + 1} x \in \mathbb{S}_{m, n + 1} \setminus \mathbb{S}_{m, n}$. Thus
  if $v (f) \neq \gamma_n$, then $f$ has an asymptotic integral
  in~$\mathbb{S}_{m, n}$. If $v (f) = \gamma_n$, then $c \log_{n + 1} x$ is an
  asymptotic integral of $f$ in $\mathbb{S}_{m, n + 1}$ with derivative in
  $\mathbb{S}_{m, n}$, where $c \in C$ is the leading coefficient of $f$. We
  also deduce that there is no pseudo-gap in $\mathbb{S}$, so the conditions
  of \Cref{def-semi-regular} hold for $(d_0,d) = ((m,n),(m, n + 1))$.
\end{proof}

\begin{lemma}
  \label{lem-transseries-contractive-hull}We have $\tmop{Cont} (\partial) = \{
  f \in \mathbb{S} \suchthat f \prec x \}$.
\end{lemma}

\begin{proof}
  We have $\Psi \geqslant v (x^{- 1})$ by definition of $\mathbb{S}$, whence
  $v (x^{- 1}) = \min \Psi$, hence the result by \Cref{lem-cont-cond}.
\end{proof}

Given $f \in \mathbb{S}^{>\mathbb{R}}$, the right composition with $f$ is the
map $\mathbb{S} \longrightarrow \mathbb{S} \: ; \: g \mapsto g \circ f$. We
say that a map $\sigma \of \mathbb{S} \longrightarrow \mathbb{S}$ satisfies a
chain rule if there is an $h \in \mathbb{T}$ such that for all $g \in
\mathbb{S}$, we have
\begin{equation}
  \partial (\sigma (g)) = h \sigma (\partial (g)) \label{eq-chain-rule}
\end{equation}
Right compositions satisfy chain rules {\cite[Proposition~6.3]{vdDMM01}}. We have a converse relation:

\begin{lemma}
  \label{lem-chain-rule-char}Let $\sigma \in 1 \text{-}
  \tmop{Aut}^{\mathcal{S}} (\mathbb{S})$ satisfy a chain rule. If $\sigma(\log s) -\log s\prec 1$ for all $s \in \mathbb{S}^>$, then $\sigma$
  is the right composition with $\sigma(x)$.
\end{lemma}

\begin{proof}
  Consider the contracting map $\phi \assign \sigma - \tmop{Id}_{\mathbb{S}}$. Let $\mathfrak{m} \in \mathfrak{M}$ with $\mathfrak{m}\succ 1$ and let $\mathfrak{n} \in \tmop{supp} \log \mathfrak{m}$. We have $1 \prec \mathfrak{n}
  \prec \mathfrak{m}$ by construction of $\log$ on $\mathbb{T}$. It follows since $(\mathbb{T},\preccurlyeq,\partial)$ is an H-asymptotic field that $\mathfrak{n}^{\dag} \asymp (\mathfrak{n}^{-1})^{\dag} \preccurlyeq (\mathfrak{m}^{-1})^{\dag} \asymp \mathfrak{m}^{\dag}$. This entails that $\log (\mathfrak{M}_{\precpreceq x}) \subseteq
  \mathfrak{M}_{\precpreceq x}$, so $\log \mathbb{S}^> 
  \subseteq \mathbb{S}$. Consider an $s \in \mathbb{S}^{>\mathbb{R}}$. We have
  \begin{eqnarray*}
    \partial (\log (\sigma (s))) - \partial (\sigma (\log s)) & = &
    \frac{\partial (\sigma (s))}{\sigma (s)} - h \sigma (\partial (\log s))\\
    & = & h \left( \frac{\sigma (\partial (s))}{\sigma (s)} - \sigma \left(
    \frac{\partial (s)}{s} \right) \right)\\
    & = & 0.
  \end{eqnarray*}
  So $c \assign \log (\sigma (s)) - \sigma (\log s) \in \tmop{Ker} (\partial)
  = C$. Our hypothesis on $\sigma$ gives $c\prec 1$, whence $c=0$, i.e. $\log (\sigma (s)) = \sigma (\log s)$. An easy induction on the directed system $(\mathbb{S}_{m,n})_{m,n \in \mathbb{N}}$ shows that the right composition with $\sigma(x)$ is the only strongly linear automorphism of $\mathbb{S}$ which commutes with $\log$ and sends $x$ to $\sigma(x)$. So it must coincide with $\sigma$.
\end{proof}

\begin{lemma}
  \label{lem-partial-Taylor}If $g \in \tmop{Cont}(\partial)$, then $\exp(g \partial)$ is a right composition. Moreover, the derivation $\partial$ is contracting on
  $\mathbb{S}$ and $\exp (\partial)$ coincides with the right composition with
  $x + 1$.
\end{lemma}

\begin{proof}
Note that $\sigma\assign\exp(g \partial)$ commutes with $g \partial$. Thus 
\[\partial \circ \sigma = \frac{1}{g} (g \partial) \circ \sigma = \frac{1}{g}\sigma \circ (g \partial) =\frac{\sigma(g)}{g} \sigma \circ \partial.\] Therefore $\sigma$ satisfies a chain rule.

  Let $s \in \mathbb{S}^>$. We have $\sigma(\log s)-\log s \sim g (\log s)'$.
Since $g \in \tmop{Cont}(\partial)$, we have $g(\log s)'\prec (\log s)' = s^{\dag}$. But $s \in \mathbb{S}^{\times}$, so $s^{\dag}\preccurlyeq x^{-1}$, so $\sigma(\log s)-\log s \prec 1$. Applying the previous lemma, we see that $\sigma$ is a right composition.

  We know that $\partial$ is contracting by
  \Cref{lem-transseries-contractive-hull}. Moreover $\exp (\partial)$ commutes
  with $\partial$, so it satisfies a chain rule. Therefore $\exp (\partial)$
  is the right composition with $\exp (\partial) (x) = x + 1 + \partial (1) +
  \partial (\partial (1)) + \cdots = x + 1$.
\end{proof}

\begin{theorem}
  \label{th-corres-trans}The set $\mathcal{P} \assign \{ x + \delta \in
  \mathbb{S} \suchthat \delta \prec x \}$ is a group under composition, and
  the map
  \begin{eqnarray*}
    \mathcal{E} \of \tmop{Cont} (\partial) & \longrightarrow & \mathcal{P}\\
    f & \longmapsto & \exp (f \partial) (x)
  \end{eqnarray*}
  is a group isomorphism between $(\tmop{Cont} (\partial), \ast, 0)$ and
  $(\mathcal{P}, \check{\circ}, x)$.
\end{theorem}

\begin{proof}
  That $\mathcal{E}$ ranges in $\mathcal{P}$ follows from the fact that $\exp
  (\tmop{Der}^{\mathcal{S}}_{\prec} (\mathbb{S})) \subseteq 1 \text{-}
  \tmop{Aut}^{\mathcal{S}} (\mathbb{S})$. Each $\exp (f \partial)$ for $f \in
  \tmop{Cont} (\partial)$ commutes with $f \partial$, so it satisfies a chain
  rule. By \Cref{lem-chain-rule-char}, this implies that $\mathcal{E}$ is
  injective. Let us show that it is a surjective morphism. Let $f \in
  \mathcal{P}$. Considering $f^{\tmop{inv}}$ if necessary, we may assume that
  $f > x$. By {\cite[Theorem~4.1]{Edgar18}}, there is a series $V \in
  \mathbb{T}^{>\mathbb{R}}$ with $V' \succ x^{- 1}$ and $V \circ f = V + 1$.
  We have
  \[ (V^{\tmop{inv}})^{\dag} = \frac{(V^{\tmop{inv}})'}{V^{\tmop{inv}}} =
     \left( \frac{1}{V' x} \right) \circ V^{\tmop{inv}} . \]
  Now $1 / (V' x) \prec 1$ so $(V^{\tmop{inv}})^{\dag} \prec 1$ by
  {\cite[Proposition~5.10]{vdH:ln}}.
  
  We claim that $V^{\dag} \preccurlyeq x^{- 1}$. Indeed, we have $f - x > x^{-
  n}$ for a certain $n > 1$. Note that $x^{n+1} \circ (x + x^{- n}) > x^{n+1} + 1$.
  The ordered group $(\mathbb{T}^{>\mathbb{R}}, \circ, x, <)$ is a growth
  order group with Archimedean centralisers as a consequence of
  {\cite[Theorem~4.7]{Bag:gog}}. So by the axiom $\textbf{GOG2}$ of
  {\cite[Section~2.1]{Bag:gog}}, for all $\varphi, \psi \in
  \mathbb{T}^{>\mathbb{R}}$ such that $\psi$ lies above all iterates of $x +
  x^{- n}$ and that $\varphi > \psi \circ \psi$, we have $\varphi \circ (x + x^{-
  n}) \circ \varphi^{\tmop{inv}} > \psi \circ (x + x^{- n}) \circ
  \psi^{\tmop{inv}}$. Here we apply this to $\psi = x^{n+1}$, and see that we must have $V \leqslant x^{2(n+1)}$. Thus $V^{\dag} \preccurlyeq
  x^{- 1}$. Now writing $V = V_0 + V_1$ where $V_0 \in \mathbb{S}$ and
  $V_1^{\dag} \succ x^{- 1}$, we have $V_0 \circ f + V_1 \circ f = V_0 + 1 +
  V_1 + 1$ where $V_0 \circ f - V_0 - 1 \in \mathbb{S}$, so we must have $V_1
  \circ f - V_1 = 0$, and thus we may assume that $V = V_0 \in \mathbb{S}$. By
  \Cref{lem-partial-Taylor}, we have
  \[ f = (V^{\tmop{inv}} \circ (x + 1)) \circ V = \sum_{i \in \mathbb{N}}
     \frac{(V^{\tmop{inv}})^{(n)} \circ V}{i!} . \]
  We have~$\frac{1}{V'} \prec x$, so the regular derivation $\mathd \assign
  \frac{1}{V'} \partial$ on $\mathbb{S}$ is contracting by
  \Cref{lem-transseries-contractive-hull}. An easy induction using the chain
  rule {\cite[Proposition~6.3]{vdDMM01}} shows that $(V^{\tmop{inv}})^{(i)}
  \circ V$ is the value of the $i$-th iterate of $\mathd$ at $x$, for all $i >
  0$. So $\exp (\mathd) (x) = f$. We see with \Cref{lem-partial-Taylor} that $\exp(\mathd)$ is the right composition with $ f$. This shows that $\mathcal{E}$ is surjective. Now $\exp \of
  \tmop{Der}^{\mathcal{S}} (\mathbb{S}) \longrightarrow 1 \text{-}
  \tmop{Aut}^{\mathcal{S}} (\mathbb{S})$ is a group morphism, and $\sigma
  \mapsto \sigma (x)$ is a group morphism $\left( 1 \text{-}
  \tmop{Aut}^{\mathcal{S}} (\mathbb{S}), \circ, \tmop{Id}_{\mathbb{S}} \right)
  \longrightarrow (\mathcal{P}, \check{\circ}, x)$, so the result follows.
\end{proof}

\begin{lemma}
  \label{lem-for-radius}For all $\delta, \varepsilon \in \mathbb{T}^{\times}$
  with $\delta, \varepsilon \prec x$ and $\delta \sim \varepsilon$ we have $1
  - x \delta^{\dag} \sim 1 - x \varepsilon^{\dag}$.
\end{lemma}

\begin{proof}
  Write $\delta = \varepsilon + \iota$ where $\iota \prec \varepsilon$. So $1
  - x \delta^{\dag} = 1 - x \varepsilon^{\dag} - x \left( 1 + \iota /
  \varepsilon \right)^{\dag}$ Recall that $\left( \Gamma, +, 0, <,
  \mathord{\:^{\dag}} \right)$ is an asymptotic couple, so we have $\left( 1 +
  \iota / \varepsilon \right)^{\dag} \sim \left( \iota / \varepsilon \right)'
  \prec \varepsilon^{\dag}$ by {\cite[axiom AC3, p 273]{vdH:mt}}, hence the
  result.
\end{proof}

\begin{theorem}
  \label{th-transseries}Two series $x + \delta, x + \varepsilon \in
  \mathcal{P} \setminus \{ x \}$ are conjugate in $\mathcal{P}$ if and only if
  $\varepsilon - \delta \prec \delta (1 - x \delta^{\dag})$.
\end{theorem}

\begin{proof}
  Note that $\exp (f \partial) (x) - x \sim f$ for all $f \in \tmop{Cont}
  (\partial)$. Recall that $v (x^{- 1}) = \min \Psi$. By
  \Cref{lem-semi-regular}, the direct limit $\mathbb{S}$ has regular asymptotic integration, so Theorem~\ref{th-main} applies. Now by Theorem~\ref{th-corres-trans}
  we can translate the conditions of conjugacy in Theorem~\ref{th-main}, as the conjunction $\varepsilon \sim \delta$ and
  $v \left( \frac{\varepsilon - \delta}{\delta^2} \right) > v \left( \left(
  \frac{x}{\delta} \right)' \right)$. Since $1 - \delta x^{- 1} \preccurlyeq
  1$, the inequality $\varepsilon - \delta \prec \delta (1 - x \delta^{\dag})$
  entails that~$\varepsilon \sim \delta$. We conclude with
  \Cref{lem-for-radius}.
\end{proof}

\begin{remark}
  \label{rem-precise}For $\delta \in \mathbb{S}$ with $\delta \prec x$, since
  $\delta^{\dag} \preccurlyeq x^{- 1}$, we have $1 - x \delta^{\dag}
  \preccurlyeq 1$. Furthermore, we have $1 - x \delta^{\dag} \prec 1$ if and
  only if $\delta^{\dag} \sim x^{- 1}$, i.e. if and only if $\delta = xh$ for
  an $h \in \mathbb{S}$ with $h^{\dag} \prec x^{- 1}$. In all other cases, the
  series $x + \delta$ and $x + \varepsilon$ are conjugate if and only if
  $\varepsilon \sim \delta$.
\end{remark}

\begin{definition}
  \label{def-subsystem}Let $\mathcal{U}= (\Lambda_d)_{d \in D}$ be a direct
  system of non-trivial ordered Abelian groups such that each $\Lambda_d$ for
  $d \in D$ is a subgroup of $\Gamma_d$ and that the morphisms $\Lambda_{d_0}
  \longrightarrow \Lambda_{d_1}$ for $d_0 \leqslant d_1$ are restrictions of
  the morphisms $\Gamma_{d_0} \longrightarrow \Gamma_{d_1}$. Set $\mathbb{U}
  \assign \underset{\longrightarrow}{\lim} _{d \in D} C \left( \!
  (\Lambda_d) \! \right)$ and $\mathcal{P}_{\mathcal{U}} \assign \{ x +
  \delta \suchthat \delta \in \mathbb{U} \wedge \delta \prec x \}$. We say
  that $\mathcal{U}$ is a {\tmstrong{{\tmem{subsystem}}}} of $\mathcal{S}$ if
  \begin{enumeratealpha}
    \item $\partial$ is
    $\mathcal{U}$-regular, and
    
    \item \label{def-subsystem-b}$\mathcal{P}_{\mathcal{U}}$ is closed under
    composition.
  \end{enumeratealpha}
\end{definition}

For instance, the direct system $(\mathfrak{M}_{0, n})_{n \in \mathbb{N}}$
corresponding to the field $\mathbb{T}_{\log}$ of logarithmic transseries, is a subsystem of $\mathcal{S}$.

\begin{proposition}
  \label{prop-subsystems}Let $\mathcal{U} \subseteq \mathcal{S}$ be a
  subsystem and let $\mathbb{U} \subseteq \mathbb{S}$ denote the corresponding
  direct limit. Then $\exp (\tmop{Cont} (\partial) \cap \mathbb{U}) (x)
  =\mathcal{P}_{\mathcal{U}}$.
\end{proposition}

\begin{proof}
  For $f \in \tmop{Cont} (\partial) \cap \mathbb{U}$, the series $\exp (f
  \partial) (x) = x + f + \frac{1}{2} ff' + \cdots$ lies in $\mathbb{U}$ by
  $\mathcal{U}$-regularity of $\partial$. Conversely, let $\delta \in
  \mathbb{U}$ with $\delta \prec x$, and let $h \in \mathbb{S}$ with $x +
  \delta = \exp (h \partial) (x)$. Assume for contradiction that $h \nin
  \mathbb{U}$. Recall that $\tmop{supp} h$ is a well-ordered subset of
  $\Gamma_d$ for a $d \in D$. So there is a least element $\gamma_0 \in
  \tmop{supp} h \setminus \Lambda_d$. Write $\mathfrak{m}_0$ for the
  corresponding element of $\mathfrak{M}_d$. Let $\varphi$ denote the element
  of $\mathbb{S}_d$ with $\tmop{supp} \varphi \assign \{ \gamma \in
  \tmop{supp} h \suchthat \gamma < \gamma_0 \}$. So $\varphi \in \mathbb{U}_d$
  and $h = \varphi + \psi$ where $\psi \sim h (\gamma_0) \mathfrak{m}_0$. By
  \Cref{lem-Lie-cont}, the series $\varepsilon \assign h \ast (- \varphi) =
  \psi + \frac{1}{2}  \llbracket \psi, - \varphi \rrbracket + \cdots$
  satisfies $\varepsilon \sim h (\gamma_0) \mathfrak{m}_0$, so $\exp
  (\varepsilon \partial) (x) - x \sim h (\gamma_0) \mathfrak{m}_0$. In
  particular, we have $\exp (\varepsilon \partial) (x) \nin \mathbb{U}$. But
  $\exp (\varepsilon \partial) (x) = \exp (- \varphi \partial) (x) \circ \exp
  (h \partial) (x) \in \mathcal{P}_{\mathcal{U}}$ by
  \Cref{def-subsystem}(\ref{def-subsystem-b}): a contradiction.
\end{proof}

\begin{corollary}
  \label{cor-log-transseries}For all $n \in \mathbb{N}$, two series $x +
  \delta, x + \varepsilon \in \mathcal{P} \cap \mathbb{T}_{0, n} \setminus \{
  x \}$ are conjugate in $\mathcal{P} \cap \mathbb{T}_{0, n + 1}$ if and only
  if $\varepsilon - \delta \prec \delta (1 - x \delta^{\dag})$. In particular,
  two series $x + \delta, x + \varepsilon \in \mathcal{P} \cap
  \mathbb{T}_{\log} \setminus \{ x \}$ are conjugate in $\mathcal{P} \cap
  \mathbb{T}_{\log}$ if and only if $\varepsilon - \delta \prec \delta (1 - x
  \delta^{\dag})$.
\end{corollary}

\begin{proof}
  Recall by \Cref{lem-semi-regular} that
  $\mathbb{T}_{\log}$ has regular asymptotic integration. The pseudo-gap $v ((\log_{n +
  1} x)')$ of $\mathbb{T}_{0, n + 1}$ is strictly below the set $v
  (\mathbb{T}_{0, n})$. So we may apply \Cref{prop-with-Poincar�} in $\tmop{Cont}(\partial \restriction \mathbb{T}_{0,n+1})$. Then, as in the proof of \Cref{th-transseries}, we translate the inequalities using the isomorphism from \Cref{prop-subsystems}.
\end{proof}

\begin{example}
  \label{ex-grid-based}Another example of subsystem of $\mathcal{S}$ is the
  direct system corresponding to flat grid-based transseries
  {\cite{Ec92,vdH:ln,Edgar:begin}}. The valued differential field of flat
  grid-based transseries has regular asymptotic integration, so one obtains
  the same conditions for conjugacy of parabolic flat grid-based transseries.
\end{example}

\begin{example}
  \label{ex-Puiseux}Consider the subsystem $\left( \frac{1}{n!} \mathbb{Z}
  \right)_{m, n \in \mathbb{N}}$ of $\mathcal{S}$. This corresponds to formal
  Puiseux series over $\mathbb{R}$. Here we have a pseudo-gap $v(x^{-1})$, so we only obtain
  the sufficient but non-necessary conditions as in
  \Cref{th-k-powered}.
\end{example}

\begin{remark}
    Our results up to and excluding \Cref{th-corres-trans} apply when replacing $\mathbb{T}$ with the larger, class-sized field of $\omega$-series \cite[Section~4.5]{BM19}. This is an totally ordered direct limit of (set-sized) fiels of generalised series indexed by all ordinals. Although it has a pseudo-gap, asymptotic integration is regular with respect to this directed system.
\end{remark}

\subsection{Formal power series with exponents in a field}

Let $C$ be an ordered field and let $\Gamma$ denote its underlying ordered
additive group. Then the field $\mathbb{K}= C \left( \! (\Gamma) \!
\right)$ is endowed with a standard derivation $\partial$ given by $\partial
(f) (c) = (c + 1) f (c + 1)$ for all $(f, c) \in \mathbb{K} \times C$. We
write each element $f$ of $\mathbb{K}$ as a formal series $f = \sum_{c \in C}
f (c) x^c$, so $\partial (f) = \sum_{c \in C} cf (c) x^{c - 1}$, and $\partial
(x) = 1$. Note that $\tmop{Cont} (\partial) = \{ f \in \mathbb{K} \suchthat f
\prec x \}$ and that the H-asymptotic field $(\mathbb{K}, \preccurlyeq,
\partial)$ is grounded in the sense of {\cite[p 326]{vdH:mt}}, i.e. the set
$\Psi$ has a maximum $v (x^{- 1})$, which is thus the pseudo-gap of
$(\mathbb{K}, \preccurlyeq, \partial)$.

\ We showed {\cite[Proposition~6.6]{Bag:mg}} that $\exp (\tmop{Cont}
(\partial)) (x)$ is the group
\[ \mathcal{P}_C \assign \{ x + \delta \suchthat \delta \in \mathbb{K} \wedge
   \delta \prec x \} \]
of parabolic series in $\mathbb{K}$ under the formal composition law of
{\cite[Section~6.1]{Bag:mg}}. In view of \Cref{prop-with-Poincar�}, the same
arguments as in the proof of \Cref{th-transseries} entail that for all
$\delta, \varepsilon \in \mathbb{K}$ with $\delta, \varepsilon \prec x$ and
$\varepsilon \sim \delta$, the series $x + \delta$ and $x + \varepsilon$ are
conjugate in $\mathcal{P}_C$ if $\varepsilon - \delta \prec \delta (1 - x
\delta^{\dag})$ and $v \left( \frac{\varepsilon - \delta}{\delta^2} \right)$
lies above the pseudo-gap of $\mathbb{K}$, i.e. if $v \left( \frac{\varepsilon
- \delta}{\delta^2} \right) + v (x^{- 1}) > 0$. This translates to
$\varepsilon - \delta \prec \delta (1 - x \delta^{\dag}), x \delta^2$. If
$\delta \prec x^{- 1}$, then as in \Cref{rem-precise}, we have $v (1 - x
\delta^{\dag}) = 0$, so $\varepsilon - \delta \prec \delta (1 - x
\delta^{\dag})$ holds because $\varepsilon \sim \delta$. If $\delta
\succcurlyeq x^{- 1}$, then $x \delta^2 \succcurlyeq \delta$ so $\varepsilon -
\delta \prec x \delta^2$ holds because $\varepsilon \sim \delta$. Therefore:

\begin{theorem}
  \label{th-k-powered}For all $\delta, \varepsilon \in \mathbb{K}$ with
  $\delta, \varepsilon \prec x$ with $\varepsilon \sim \delta$, the series $x
  + \delta$ and $x + \varepsilon$ are conjugate in $\mathcal{P}_C$ if $\delta
  \succcurlyeq x^{- 1}$ and $\varepsilon - \delta \prec \delta (1 - x
  \delta^{\dag})$ or $\delta \prec x^{- 1}$ and $\varepsilon - \delta \prec x
  \delta^2$.
\end{theorem}

\subsection{A case of resonance}\label{subsection-resonnance}

We conclude by giving a simple example of resonance. Consider the field
$\mathbb{K}$ above for a given ordered field $C$. For all $c < 0$ in $C$ and
$\delta \in \mathbb{K}$ with $v (\delta) = v (x^c)$, we write $f_{\delta}
\assign x + 1 + \delta$, and set $f_0 \assign x + 1$. The pseudo-gap of
$\mathbb{K}$ is $v (x^{- 1})$. Thus for $\delta \in \mathbb{K}$ with $v
(\delta) = v (x^{- 1})$, the series $f_0$ and $f_{\delta}$ are not
asymptotically conjugate, whence not conjugate. However, for $c \in (- 1, 0)$
in $C$, and $\varepsilon \assign x^c$, the series $f_0$ and $f_{\varepsilon}$
are asymptotically conjugate, and the approximative conjugacy method of
\Cref{lem-approximate-conjugacy} (translated via the exponential map
$\mathcal{E} \of \tmop{Cont} (\partial) \longrightarrow \mathcal{P}_C$) gives
an asymptotic conjugating element $\varphi \assign x + \frac{1}{c + 1} x^{c +
1}$. Using formal Taylor expansions (see {\cite[Section~6.1]{Bag:mg}}) of up
to order $2$, one obtains
\[ \varphi \circ f_{\varepsilon} \circ \varphi^{\tmop{inv}} = \varphi \circ
   \left( \varphi^{\tmop{inv}} + 1 + x^c - \frac{c}{c + 1} x^{2 c} + \cdots
   \right) = f_{\iota} \]
for a $\iota \sim \frac{1}{c + 1} x^{2 c}$. If $c \in \left( - 1, - 1 / 2
\right)$, then $v (x^{2 c}) = v (f_0 - f_{\iota}) = v \left(
\frac{f_{\varepsilon} - x - (f_{\iota} - x)}{(f_{\varepsilon} - x)^2} \right)$
lies above the pseudo-gap of $\mathbb{K}$, so $f_0$ and $f_{\iota}$ are
conjugate, whence $f_0$ and $f_{\varepsilon}$ are conjugate. In particular, in
contrast with \Cref{cor-initial}, the set of series $\delta \prec 1$ for which
$x + 1$ and $x + 1 + \delta$ are conjugate is not downward closed for
$\preccurlyeq$. It is known that 

\begin{acknowledgments*}
  We thank Daniel Panazzolo and Jean-Philippe Rolin for their answers to our
  questions.
\end{acknowledgments*}

\end{document}